\documentclass[12pt,reqno]{amsart}

\usepackage[a4paper, centering, scale=0.75]{geometry}
\usepackage{thmtools, thm-restate}
\usepackage{enumerate}
\usepackage{slashed}
\usepackage{txfonts,dsfont}
\usepackage[breaklinks,colorlinks]{hyperref}
\usepackage{nameref}
\usepackage{cleveref}

\makeatletter
\newcommand{\autorefcheckize}[1]{%
  \expandafter\let\csname @@\string#1\endcsname#1%
  \expandafter\DeclareRobustCommand\csname relax\string#1\endcsname[1]{%
    \csname @@\string#1\endcsname{##1}\wrtusdrf{##1}}%
  \expandafter\let\expandafter#1\csname relax\string#1\endcsname
}
\makeatother

\declaretheorem[numberwithin=section]{theorem}
\declaretheorem[sibling=theorem, name=Lemma]{lem}
\declaretheorem[sibling=theorem, name=Corollary]{cor}

\declaretheorem[sibling=theorem, name=Remark]{rem}

\declaretheorem[numberwithin=section, name=Example]{eg}

\numberwithin{equation}{section}

\newcommand{\norm}[1]{\left\lVert#1\right\rVert}
\newcommand{\abs}[1]{\left\lvert#1\right\rvert}
\newcommand{\set}[1]{\left\{#1\right\}}

\newcommand*{\Rmn}[1]{\uppercase\expandafter{\romannumeral#1}}
\newcommand*{\To}{\longrightarrow}

\newcommand*{\dif}{\mathop{}\!\mathrm{d}}

\DeclareMathOperator{\esssup}{ess\, \sup}

\allowdisplaybreaks

\begin{document}

\title[Brouwer degree for Kazdan-Warner equations on a connected finite graph]{Brouwer degree for Kazdan-Warner equations on a connected finite graph}


\author[Linlin Sun]{Linlin Sun}
\address{School of Mathematics and Statistics, Wuhan University, Wuhan 430072, China}
\address{Hubei Key Laboratory of Computational Science, Wuhan University, Wuhan, 430072, China}
\email{sunll@whu.edu.cn}

\author[Liuquan Wang]{Liuquan Wang}
\address{School of Mathematics and Statistics, Wuhan University, Wuhan 430072, China}
\email{wanglq@whu.edu.cn; mathlqwang@163.com}

\thanks{This work is partially supported by the National Natural Science Foundation of China (Grant Nos. 11801420, 11801424, 11971358).}



\date{\today}


\begin{abstract}

We study Kazdan-Warner equations on a connected finite graph via the method of the degree theory. Firstly, we prove that all solutions to the Kazdan-Warner equation with nonzero prescribed function are uniformly  bounded and the Brouwer degree is well defined. Secondly, we compute the Brouwer degree case by case. As consequences, we give new proofs of some known existence results for the Kazdan-Warner equation on a connected finite graph.

\end{abstract}

\maketitle

\section{Introduction}
Let $\Sigma$ be a closed Riemann surface, $h$ and $f$ two smooth functions on $\Sigma$. The Kazdan-Warner equation reads as
\begin{align}\label{eq:KW-surface}
    -\Delta u=he^{u}-f,
\end{align}
where $\Delta$ is the Laplace-Beltrami operator.
It comes from the prescribed Gaussian curvature problem  \cite{KazWar74curvature,ChaYan87prescribing,CheDin87scalar},
and also appears in various contexts such as the abelian Chern-Simons-Higgs models  \cite{RicTar00vortices,Nol03nontopological,CafYan95vortex}.

The existence of solutions to the Kazdan-Warner equation has been studied in recent decades.
Denote by $\dif\mu_{\Sigma}$ the area element of $\Sigma$. If $\int_{\Sigma}f\dif\mu_{\Sigma}=0$ and $h\not\equiv0$, then the Kazdan-Warner equation  \eqref{eq:KW-surface} is solvable \cite{KazWar74curvature} if and only if $h$ changes sign and
\begin{align*}
    \int_{\Sigma}he^{\phi}\dif\mu_{\Sigma}<0,
\end{align*}
where $\phi$ is the unique solution to
\begin{align*}
    -\Delta\phi=\dfrac{\int_{\Sigma}f\dif\mu_{\Sigma}}{\int_{\Sigma}1\dif\mu_{\Sigma}}-f,\quad\int_{\Sigma}\phi\dif\mu_{\Sigma}=0.
\end{align*}
If $\int_{\Sigma}f\dif\mu_{\Sigma}\neq0$, then the Kazdan-Warner equation \eqref{eq:KW-surface} can be reduced to the following mean field equation
\begin{align}\label{eq:mf}
    -\Delta u=\rho\left(\dfrac{he^{u}}{\int_{\Sigma}he^{u}\dif\mu_{\Sigma}}-\dfrac{1}{\int_{\Sigma}1\dif\mu_{\Sigma}}\right)
\end{align}
where $\rho\in\mathbb{R}\setminus\set{0}$. Many partial existence results of the mean field equation have been obtained for both noncritical and critical cases, see for examples Struwe and Tarantello \cite{StruTar98multivortex}, Ding, Jost, Li and Wang \cite{DinJosLiWan99existence}, Chen and Lin \cite{CheLin03topological}, Djadli \cite{Dja08existence} and the references therein.

If the prescribed function $h$ is a positive function and $\rho\in\mathbb{R}\setminus 8\pi\mathbb{N}^*$, then every solution $u$ with $\int_{\Sigma}u\dif\mu_{\Sigma}=0$  to the mean field equation \eqref{eq:mf} is uniformly bounded. One can define the Leray-Schauder degree for equation \eqref{eq:mf} as follows (cf. \cite[p. 422]{Li99harnack}). Given $\alpha\in(0,1)$, consider
\begin{align*}
    X_{\alpha}=\set{u\in C^{2,\alpha}\left(\Sigma\right):\int_{\Sigma}u\dif\mu_{\Sigma}=0},
\end{align*}
and introduce a compact operator $K_{\rho,h}:X_{\alpha}\To X_{\alpha}$ by
\begin{align*}
    K_{\rho,h}(u)=\rho\left(-\Delta\right)^{-1}\left(\dfrac{he^{u}}{\int_{\Sigma}he^{u}\dif\mu_{\Sigma}}-\dfrac{1}{\int_{\Sigma}1\dif\mu_{\Sigma}}\right).
\end{align*}
The Leray-Schahder degree is defined by
\begin{align*}
    d_{\rho}=\lim_{R\to+\infty}\deg\left(\mathrm{Id}-K_{\rho,h}, B_{R}^{X_{\alpha}},0\right)
\end{align*}
 which is independent of $\alpha$ and $h$. Here $B_{R}^{X_{\alpha}}$ stands for the ball with center at the origin and radius $R$ in the Banach space $X_{\alpha}$ equipped with the $C^{2,\alpha}$-norm. Li \cite[p. 422]{Li99harnack} pointed out that $d_{\rho}$ should be  determined by the Euler number $\chi\left(\Sigma\right)$ of $\Sigma$.
Chen and Lin \cite[Theorem 1.2]{CheLin03topological} proved that
\begin{align*}
    d_{\rho}=\binom{k-\chi\left(\Sigma\right)}{k},
\end{align*}
where  $\rho\in\left(8k\pi,8(k+1)\pi\right)$ and $k\in\mathbb{N}$. As a consequence, if the genus of $\Sigma$ is greater than zero, then the mean field equation \eqref{eq:mf} with positive prescribed function $h$ always possesses a solution provided that $\rho$ is not a multiple  of $8\pi$.

In this paper, we  consider the following Kazdan-Warner equation on a connected finite graph $G=(V,E)$:
 \begin{align}\label{eq:KW}
     -\Delta u(x)=h(x)e^{u(x)}-c,\quad x\in V,
 \end{align}
 where $\Delta$ is the Laplace operator on $G$ (see \eqref{mu-Laplace}), $h$ is a real function on $V$ and $c$ is a real number.
 This equation was studied by several mathematicians (cf. \cite{GriLinYan16kazdan,Ge17kazdan,LiuYan20multiple,KelSch18kazdan,HuaLinYau20existence,GeJia18Kazdan}). For example, utilizing the variational method, Grigor'yan, Lin and Yang \cite[Theorems 1-3]{GriLinYan16kazdan} obtained the following discrete analog of that of Kazdan and Warner \cite{KazWar74curvature}:
 \begin{itemize}
     \item when $c=0$, \eqref{eq:KW} has a solution if and only if $h\equiv0$ or $h$ changes sign and $\int_{V}h\dif\mu<0$;
     \item when $c>0$, \eqref{eq:KW} is solvable if and only if $\max_{V}h>0$;
     \item  when $c<0$, if \eqref{eq:KW} has a solution, then $\int_Vh\dif \mu<0$, and in this case, there exists a constant  $c_{h}\in[-\infty,0)$ depending on $h$ such that \eqref{eq:KW} has a solution if $c\in(c_h,0)$, but has no solution for any $c<c_h$.
 \end{itemize}
Grigor'yan, Lin and Yang \cite[Theorem 4]{GriLinYan16kazdan} pointed out that $c_h=-\infty$ if $\min_{V}h<\max_V h\leq0$. Ge \cite{Ge17kazdan} proved that $c_{h}<-\infty$ if $h$ changes sign and obtained a solution for $c=c_h$. Recently, Liu and Yang \cite{LiuYan20multiple} studied the following Kazdan-Warner equation
 \begin{align}\label{eq:K-lambda}
    -\Delta u=K_{\lambda}e^{u}-\kappa
\end{align}
where $\int_{V}\kappa\dif\mu<0$, $K_{\lambda}=K+\lambda$ and $\min_{V}K<\max_{V}K=0$. They obtained the following discrete analog of that of Ding and Liu \cite{DinLiu95note}: there exists a $\lambda^*\in(0,-\min_{V}K)$ such that \eqref{eq:K-lambda} has a unique solution if $\lambda\leq0$, at least two distinct solutions if $0<\lambda<\lambda^*$, at least one solution if $\lambda=\lambda^*$, and no solution if $\lambda>\lambda^*$.

 Our aim is to extend the results of Chen and Lin \cite{CheLin03topological} to graphs. We shall prove that every solution to the Kazdan-Warner equation \eqref{eq:KW} is uniformly bounded whenever $h\not\equiv0$. Consequently, the Brouwer degree $d_{h,c}$ for \eqref{eq:KW} can be well defined. We will give the exact formula for the Brouwer degree $d_{h,c}$. Meanwhile, we will use the degree theory to recover some known existence results.

The remaining part of this paper is briefly organized as follows. In Section \ref{sec:main} we recall some notions on a graph and state our main results.  In Section \ref{sec:preliminaries} we recall some basic facts regarding functions on a connected finite  graph. In Section \ref{sec:blowup} we study the blow-up behavior for the Kazdan-Warner equation and state a discrete analog of that of Brezis and Merle's result in \cite{BreMer91uniform}. We will prove a compactness result for the Kazdan-Warner equation on a connected finite graph. In particular, we give a proof of \autoref{main:a-priori}. In Section \ref{sec:degree} we compute the Brouwer degree for the  Kazdan-Warner equation case by case (\autoref{main:degree}). In Section  \ref{sec:existence} we shall give new proofs of several known existence results by using the degree theory (\autoref{main:cor1} and \autoref{main:cor2}).  Hereafter we do not distinguish sequence and subsequence unless necessary. Moreover, we use the capital letter $C$ to denote some uniform constants which are independent of the special solutions and not necessarily the same at each appearance.

 \section{Settings and Main results}\label{sec:main}
Throughout this paper, $G=(V,E)$ is assumed to be a  connected finite graph with vertex set $V$ and  edge set $E$. The edges on the graph are allowed to be weighted. Let $\omega:V\times V\To\mathbb{R}$ be a weight function in the sense that  $\omega_{xy}=\omega_{yx}\geq0$ and
\begin{align*}
\omega_{xy}>0\ \Longleftrightarrow\  xy\in E.
\end{align*}
$G$ is connected means that for every $x,y\in V$ there exist $x_i\in V$ such that $x=x_1, y=x_m$ and
\begin{align*}
\omega_{x_ix_{i+1}}>0,\quad i=1,\dotsc,m-1.
\end{align*}
We say that $G$ is finite if the number of vertices is finite. Denote by $V^{\mathbb{R}}$ the set of real functions on $V$.  Let $\mu$ be a positive function (vertex measure) on $V$ and define the ($\mu$-)Laplace operator $\Delta$ by
\begin{align}\label{mu-Laplace}
\Delta u(x)\coloneqq\dfrac{1}{\mu_x}\sum_{y\in V}\omega_{xy}\left(u(y)-u(x)\right),\quad x\in V, \quad u\in V^{\mathbb{R}}.
\end{align}
For any function $f\in V^{\mathbb{R}}$, an integral of $f$ over $V$ is defined by
\begin{align*}
\int_Vf\dif\mu\coloneqq\sum_{x\in V}f(x)\mu_x.
\end{align*}
We have the following Green formula:
\begin{align}\label{Green}
\int_{V}\Delta u v\dif\mu=-\int_{V}\Gamma\left(u,v\right)\dif\mu,
\end{align}
where $\Gamma$ is the  associated gradient form:
\begin{align}\label{gradient-eq}
\Gamma(u,v)(x)\coloneqq\dfrac{1}{2\mu_x}\sum_{y\in V}\omega_{xy}\left(u(y)-u(x)\right)\left(v(y)-v(x)\right).
\end{align}
Let
\begin{align*}
\abs{\nabla u}(x)\coloneqq\sqrt{\Gamma(u,u)(x)}.
\end{align*}
Denote by $L^P\left(V\right)$ the space of all functions $f\in V^{\mathbb{R}}$ with finite norm  $\norm{f}_{L^p\left(V\right)}$ which is defined by
\begin{align*}
\norm{f}_{L^p(V)}\coloneqq\begin{cases}
\left(\int_{V}\abs{f}^p\dif\mu\right)^{1/p},&1\leq p<\infty,\\
\esssup_{V}\abs{f},&p=\infty.
\end{cases}
\end{align*}
We also consider the Sobolev space $W^{1,p}\left(V\right)$ which consist of all functions $f\in V^{\mathbb{R}}$ with finite norm $\norm{f}_{W^{1,p}(V)}$ which is defined by
\begin{align*}
    \norm{f}_{W^{1,p}(V)}\coloneqq \norm{f}_{L^p\left(V\right)}+\norm{\abs{\nabla f}}_{L^p\left(V\right)}.
\end{align*}

For every $h, f\in V^{\mathbb{R}}$, we consider the following functional
\begin{align*}
J_{h,f}(u)=\int_{V}\left(\dfrac12\abs{\nabla u}^2+fu-he^{u}\right)\dif\mu,\quad u\in W^{1,2}\left(V\right).
\end{align*}
The critical points of $J_{h,f}$ are exactly the solutions to the following (generalized) Kazdan-Warner equation:
\begin{align}\label{eq:KW-hf}
    -\Delta u(x)=h(x)e^{u(x)}-f(x),\quad x\in V.
\end{align}
We say that $u$ is stable if
\begin{align*}
\int_{V}\left(\abs{\nabla\xi}^2-he^{u}\xi^2\right)\dif\mu\geq0,\quad\forall \xi\in V^{\mathbb{R}},
\end{align*}
and $u$ is strictly stable if the equality holds only if $\xi\equiv0$.

 The first main theorem is the following a priori estimate.
\begin{theorem}\label{main:a-priori}
Let $G=(V,E)$ be a connected finite graph with weight $\omega$ and measure $\mu$. Assume that $h\in V^{\mathbb{R}}$ and $c\in\mathbb{R}$ satisfy:
  \begin{enumerate}[$1)$]
      \item if $c$ is positive, then $h$ is positive somewhere;
      \item if $c$ equals to zero, then $h$ changes sign and the integral of $h$ over $V$ is negative;
      \item if $c$ is negative, then $h$ is negative somewhere.
  \end{enumerate}
  Then there exists a constant $C$ depending only on $h,c,G,\omega$ and $\mu$ such that every solution $u$ to the Kazdan-Warner equation \eqref{eq:KW} satisfies
\begin{align*}
    \max_{V}\abs{u(x)}\leq C.
\end{align*}
\end{theorem}

\begin{rem}
The conditions $1)-3)$ mentioned in \autoref{main:a-priori} are necessary conditions to solve the Kazdan-Warner equation \eqref{eq:KW}.
\end{rem}

Assume that $h,f\in V^{\mathbb{R}}$ satisfy:
\begin{enumerate}[1)']
    \item if $\int_{V}f\dif\mu>0$, then $\max_{V}h>0$;
    \item if $\int_{V}f\dif\mu=0$, then $\int_{V}he^{\phi}\dif\mu<0<\max_{V}h$;
    \item if $\int_{V}f\dif\mu<0$, then $\min_{V}h<0$,
\end{enumerate}
where $\phi$ is the unique solution to
\begin{align*}
    -\Delta\phi=\dfrac{\int_{V}f\dif\mu}{\int_{V}1\dif\mu}-f,\quad\min_{V}\phi=0.
\end{align*}
Consider a map
\begin{align*}
F_{h,f}: L^{\infty}\left(V\right)\To L^{\infty}\left(V\right),\quad u\mapsto F_{h,f}(u)\coloneqq -\Delta u+f-he^{u}.
\end{align*}
We denote by $B_R$ the ball with center at the origin and radius $R$ in $L^{\infty}\left(V\right)$. Notice that if $u$ solves \eqref{eq:KW-hf} then
\begin{align*}
    -\Delta\left(u-\phi\right)=he^{\phi}e^{u-\phi}-\dfrac{\int_{V}f\dif\mu}{\int_{V}1\dif\mu}.
\end{align*}
Applying \autoref{main:a-priori} we conclude that there is no solution on the boundary $\partial B_{R}$ for $R$ large. Hence, the Brouwer degree
\begin{align*}
\deg\left(F_{h,f}, B_{R}, 0\right)
\end{align*}
is well defined for $R$ large. According to the homotopic invariance, $\deg\left(F_{h,f}, B_{R}, 0\right)$ is independent of $R$. Let
\begin{align*}
d_{h,f}\coloneqq\lim_{R\to+\infty}\deg\left(F_{h,f}, B_{R}, 0\right).
\end{align*}
If $J_{h,f}$ is a Morse function, i.e., every critical point of $J_{h,f}$ is nondegenerate, then
\begin{align*}
    \deg\left(F_{h,f}, B_{R}, 0\right)=\sum_{u\in B_{R}, F_{h,f}(u)=0}\det\left(D F_{h,f}(u)\right)
\end{align*}
whenever $\partial B_{R}\cap F_{h,f}^{-1}\left(\set{0}\right)=\emptyset$.
For more details about the Brouwer degree and its various properties we refer the reader to Chang \cite[Chapter 3]{Cha05methods}.

The second main theorem is the following
\begin{theorem}\label{main:degree}Let $G=(V,E)$ be a  connected finite graph and $h, c$ as in \autoref{main:a-priori}. Then
  \begin{align*}
      d_{h,c}=\begin{cases}
      -1,&c\geq0;\\
      1,&c<0\ \text{and}\  \max_{V}h\leq0;\\
      0,&c<0\ \text{and}\ \max_{V}h>0.
      \end{cases}
  \end{align*}
\end{theorem}

The Kronecker existence implies that there exists at least one solution if the Brouwer degree is nonzero.
As applications, we obtain several existence results mentioned in the introduction.

\begin{cor}[cf. \cite{GriLinYan16kazdan,Ge17kazdan}]\label{main:cor1}Let $G=(V,E)$ be a connected finite graph and $h\not\equiv0$.
\begin{enumerate}[$(1)$]
\item If $c>0$, then \eqref{eq:KW} is solvable if and only if $\max_{V}h>0$.
\item If $c=0$, then \eqref{eq:KW} is solvable if and only if $h$ changes sign and $\int_{V}h\dif\mu<0$.
\item If $c<0$ and $h\leq0$, then \eqref{eq:KW} has  a unique (strict global minimum) solution.
    \item If $c<0$ and $\int_{V}h\dif\mu<0<\max_{V}h$, then there exists a constant $c_h\in\left(-\infty,0\right)$ such that \eqref{eq:KW} has at least two distinct solutions for $c_h<c<0$, at least a (stable) solution for $c=c_{h}$, and no solution for $c<c_{h}$.
    \end{enumerate}
\end{cor}
\begin{rem}
Checking the proof of \cite[Theorem 1.1]{Ge17kazdan}, one conclude that
\begin{align*}
    c_{h}\geq-\dfrac{C\max_{V}\abs{h}}{\max_{V}h}
\end{align*}
if $\max_{V}h>0$. The multiplicity solutions to  the  Kazdan-Warner equation in the negative case can also be obtained by using the minimax method (cf. \cite{LiuYan20multiple}).
\end{rem}

\begin{cor}[cf. \cite{LiuYan20multiple}]\label{main:cor2}Let $G=(V,E)$ be a connected finite graph.
There exists a constant $\lambda^*\in\left(0,-\min_{V}K\right)$ satisfying:
\begin{enumerate}[$(1)$]
\item if $\lambda\leq0$, then \eqref{eq:K-lambda} has a unique (strict  global  minimum) solution;
\item if $0<\lambda<\lambda^*$, then \eqref{eq:K-lambda} has at least two distinct solutions;
\item if $\lambda=\lambda^*$, then \eqref{eq:K-lambda} has at least a stable  solution;
\item if $\lambda>\lambda^*$, then \eqref{eq:K-lambda} has no solution.
\end{enumerate}
\end{cor}

\section{Preliminaries}\label{sec:preliminaries}

In this section, we provide discrete versions of the strong maximum principle, the elliptic estimate,  Kato's inequality and the sub- and super-solutions principle.

We begin with the following strong maximum principle.
\begin{lem}[Strong maximum principle]\label{lem:maximum} If $u$ is not a constant function, then there exists $x_1\in V$ such that
\begin{align*}
u\left(x_1\right)=\max_{V}u,\quad\Delta u\left(x_1\right)<0.
\end{align*}
\end{lem}
\begin{proof}Choose $x,y\in V$ such that
\begin{align*}
u(x)=\max_{V}u,\quad u(y)=\min_{V}u.
\end{align*}
Since $G$ is connected, there exist  $x_i\in V$ such that $x=x_1, y=x_m$ and
\begin{align*}
\omega_{x_{i}x_{i+1}}>0,\quad i=1,\dotsc,m-1.
\end{align*}
Since $u$ is not a constant function, we have
\begin{align*}
u(x)>u(y).
\end{align*}
Thus there exists some $1\leq i \leq m-1$ such that $u(x_1)=\cdots =u(x_i)>u(x_{i+1})$. Without loss of generality, we may assume that $u\left(x_1\right)>u(x_2)$. Then $u\left(x_1\right)=\max_{V}u$ and
\begin{align*}
\Delta u\left(x_1\right)=&\dfrac{1}{\mu_{x_1}}\sum_{y\in V}\omega_{x_1y}\left(u(y)-u\left(x_1\right)\right)\\
\leq&\dfrac{1}{\mu_{x_1}}\omega_{x_1x_2}\left(u(x_2)-u\left(x_1\right)\right)\\
<&0.
\end{align*}
We complete the proof.
\end{proof}

Since $G$ is a finite graph, all of the spaces $L^p\left(V\right)$ and $W^{1,p}\left(V\right)$ with $1\leq p\leq\infty$  are exactly $V^{\mathbb{R}}$,  a finite dimensional linear space. Note that every two norms on $V^{\mathbb{R}}$ are equivalent. Denote by $\norm{\cdot}$ the norm of $V^{\mathbb{R}}$ for convenience. Set
\begin{align*}
V_0^{\mathbb{R}}=\set{u\in V^{\mathbb{R}} : \int_{V}u\dif\mu=0}.
\end{align*}
Then all of $\max_{V}\abs{\Delta u}, \max_{V}u-\min_{V}u$ and $\max_{V}\abs{\nabla u}$ are norms of $u$ on $V_0^{\mathbb{R}}$. Consequently, we have the following elliptic estimate.
\begin{lem}[Elliptic estimate]\label{lem:elliptic}
There is a positive constant $C$ such that for all $u\in V^{\mathbb{R}}$,
\begin{align*}
\max_{V}u-\min_{V}u\leq C\max_{V}\abs{\Delta u}.
\end{align*}
\end{lem}

For any function $f\in V^\mathbb{R}$, we denote by $f^+=\max\set{f,0}$ and $f^-=\left(-f\right)^+$.
For any set $A$ we define the indicator
$$\chi_A(t)\coloneqq\left\{\begin{array}{ll}
1 & t\in A, \\
0 & t\notin A.
\end{array}\right.$$
The following inequality is useful.
\begin{lem}[Kato's inequality]\label{lem:kato}
\begin{align*}
\Delta u^+\geq\chi_{\set{u>0}}\Delta u.
\end{align*}
\end{lem}
\begin{proof}By definition,
\begin{align*}
\Delta u^+(x)=\dfrac{1}{\mu_{x}}\sum_{y\in V}\omega_{xy}\left(u^+(y)-u^+(x)\right).
\end{align*}
If $u(x)>0$, then $u^+(x)=u(x)$ and
\begin{align*}
\Delta u^+(x)\geq\dfrac{1}{\mu_{x}}\sum_{y\in V}\omega_{xy}\left(u(y)-u(x)\right)=\Delta u(x).
\end{align*}
If $u(x)\leq0$, then $u^+(x)=0$ and
\begin{align*}
\Delta u^+(x)=\dfrac{1}{\mu_{x}}\sum_{y\in V}\omega_{xy}u^+(y)\geq0. 
\end{align*}
We obtain the desired inequality.
\end{proof}

Let $f:V\times\mathbb{R}\To\mathbb{R}$ be a smooth function. We say that $\phi$ is a sub-solution (super-solution) to
\begin{align}\label{eq:quasilinear}
    -\Delta u(x)=f\left(x,u(x)\right),\quad\forall x\in V,
\end{align}
if  $\Delta\phi(x)+f\left(x,\phi(x)\right)\geq(\leq0)$ for all $x\in V$. Denote
\begin{align*}
    J(u)=\int_{V}\left(\dfrac12\abs{\nabla u}^2-F\left(\cdot,u(\cdot)\right)\right)\dif\mu,
\end{align*}
where $\frac{\partial F}{\partial u}=f$.
Then we have the following sub- and  super-solutions principle.

\begin{lem}[sub- and  super-solutions  principle]\label{lem:sub-super}
Assume $\phi$ and $\psi$ are sub-solution and super-solution to \eqref{eq:quasilinear} respectively with $\phi\leq\psi$. Then any minimizer of $J$ in $\set{u\in V^{\mathbb{R}}:\phi\leq u\leq\psi}$ solves \eqref{eq:quasilinear}.
\end{lem}

\begin{proof}Without loss of generality, assume $\phi\not\equiv\psi$. Then
\begin{align*}
    -\Delta\left(\psi-\phi\right)(x)\geq f\left(x,\psi(x)\right)-f\left(x,\phi(x)\right),\quad x\in V.
\end{align*}
We claim that $\psi>\phi$. Let $u$ be a minimizer of $J$ in $\set{u\in V^{\mathbb{R}}:\phi\leq u\leq\psi}$. If $u\left(x_0\right)=\phi\left(x_0\right)$ for some $x_0\in V$, then $u\equiv\phi$. In fact, since $u-\phi\geq0$ and $\min_{V}\left(u-\phi\right)=0$, if $u\not\equiv\phi$, applying  \autoref{lem:maximum}, then there exists some $x_1 \in V$ such that
\begin{align}\label{add-x1-eq}
u\left(x_1\right)-\phi\left(x_1\right)=0,\quad \Delta(u-\phi)\left(x_1\right)>0.
\end{align}
On the one hand, since $u$ is a minimizer of $J$, we have
\begin{equation}\label{add-xi-eq}
    \begin{split}
        0\leq&\dfrac{\dif}{\dif t}\left.J\left(u+t\delta_{x_1}\right)\right\rvert_{t=0} \\
    =&\int_{V}\left(-\Delta u-f\left(\cdot,u(\cdot)\right)\right)\delta_{x_1}\dif\mu  \\
    =&-\Delta u\left(x_1\right)-f\left(x_1,u(x_1)\right).
    \end{split}
\end{equation}
On the other hand, by \eqref{add-x1-eq} and \eqref{add-xi-eq},  we have
\begin{align*}
    0<&\Delta\left(u-\phi\right)\left(x_1\right)\\
    \leq&-f\left(x_1,u(x_1)\right)+f\left(x_1,\phi(x_1)\right)\\
    =&0,
\end{align*}
which is a contradiction.

 If $u\equiv\phi$, then \eqref{add-xi-eq} implies that $\phi$ is also a super-solution and thus $u$ solves \eqref{eq:quasilinear}.  Similarly if $u(x)=\psi(x)$ for some $x\in V$ then $u\equiv\psi$, and therefore $u$ is a solution. If $\phi(x)<u(x)<\psi(x)$ for any $x\in V$, then for every $\eta\in V^{\mathbb{R}}$, since $u$ is a minimizer,
\begin{align*}
    0=\dfrac{\dif}{\dif t}\left.J\left(u+t\eta\right)\right\rvert_{t=0}
    =\int_{V}\left(-\Delta u-f\left(\cdot,u(\cdot)\right)\right)\eta\dif\mu.
\end{align*}
Thus $u$ solves \eqref{eq:quasilinear}.
\end{proof}

\section{Blow-up analysis} \label{sec:blowup}

First we state the following discrete analog of that of Brezis and Merle's result \cite{BreMer91uniform}.
\begin{theorem}\label{thm:alternative}Let $G=(V,E)$ be a  connected finite graph.
Let $u_n\in V^{\mathbb{R}}$ be a sequence  of solutions to
\begin{align*}
-\Delta u_n(x)=h_n(x)e^{u_n(x)}-c_n,\quad x\in V
\end{align*}
where $h_n\in V^{\mathbb{R}}$ and
$c_n\in\mathbb{R}$ satisfy
\begin{align*}
\lim_{n\to\infty}h_n(x)=h(x),\quad\forall x\in V
\end{align*}
and
\begin{align*}
    \lim_{n\to\infty}c_n=c.
\end{align*}
Then after passing to a subsequence, we have the following alternatives:
\begin{enumerate}[$(1)$]
\item either $u_n$ is uniformly bounded, or
\item $u_n$ converges uniformly to $-\infty$, or
\item there exists $x_0\in V$ such that $u_n\left(x_0\right)$ converges to $+\infty$ and $h\left(x_0\right)=0$. Moreover, $u_n$ is uniformly bounded from below in $V$ and above in $\set{x\in V: h(x)>0}$.
\end{enumerate}
\end{theorem}
\begin{proof}
If $u_n$ is uniformly bounded from above, then $\Delta u_n$ is uniformly bounded. Applying   \autoref{lem:elliptic},
\begin{align*}
\max_{V}u_n-\min_{V}u_n\leq C.
\end{align*}
If $\min_{V}u_n$ is uniformly bounded from below, then we obtain the first alternative. If $\liminf\limits_{n\to\infty}\min_{V}u_n=-\infty$, then we obtain the second alternative.

If $\limsup\limits_{n\to\infty}u_{n}=+\infty$, then without loss of generality we may assume for some $x_0\in V$,
\begin{align*}
0<u_n\left(x_0\right)=\max_{V}u_n\to+\infty
\end{align*}
as $n\to\infty$. According to \autoref{lem:kato}, we get
\begin{align*}
-\Delta u_n^{-}=&-\Delta\left(-u_n\right)^{+}\\
\leq&-\chi_{\set{-u_n>0}}\Delta\left(-u_n\right)\\
=&\chi_{\set{u_n<0}}\left(c_n-h_ne^{u_n}\right)\\
\leq&c_n^++h_n^{-}.
\end{align*}
Thus
\begin{align*}
\norm{\Delta u_n^{-}}_{L^1\left(V\right)}=&\int_{V}\abs{\Delta u_n^{-}}\dif\mu\\
=&\int_{\set{\Delta u_n^{-}\geq 0}}\Delta u_n^-\dif\mu-\int_{\set{\Delta u_n^{-}<0}}\Delta u_n^-\dif\mu\\
=&-2\int_{\set{\Delta u_n^{-}<0}}\Delta u_n^-\dif\mu\\
\leq&2\int_{\set{\Delta u_n^-<0}}\left(c_n^++h_n^-\right)\dif\mu\\
\leq&C.
\end{align*}
Applying   \autoref{lem:elliptic},
\begin{align*}
\max_{V}u_n^{-}=\max_{V}u_n^{-}-\min_{V}u_n^{-}\leq C.
\end{align*}
Thus $u_n$ is uniformly bounded from below.

Since $u_n$ is uniformly bounded from below, we have for every $x_1\in V$
\begin{align*}
h_n\left(x_1\right)e^{u_n\left(x_1\right)}-c_n=&-\Delta u_n\left(x_1\right)\\
=&\dfrac{1}{\mu_{x_1}}\sum_{y\in V}\omega_{x_1y}\left(u_n\left(x_1\right)-u_n(y)\right)\\
\leq&C\left(u_n\left(x_1\right)+1\right),
\end{align*}
which implies
\begin{align*}
h_n\left(x_1\right)\leq C \left(u_n\left(x_1\right)+1\right)e^{-u_n\left(x_1\right)}.
\end{align*}
Taking $n\to\infty$, we deduce that
\begin{align*}
h\left(x_1\right)\leq0
\end{align*}
whenever $\limsup\limits_{n\to\infty}u_n\left(x_1\right)=+\infty$. In other words, $u_n$ is uniformly bounded in $\set{x\in V: h(x)>0}$.

Now we prove that $h\left(x_0\right)=0$. The above argument yields that $h\left(x_0\right)\leq0$. It suffices to prove that  $h\left(x_0\right)\geq0$. Applying the maximum principle, we have
\begin{align*}
h_n\left(x_0\right)e^{u_n\left(x_0\right)}-c_n=-\Delta u_n\left(x_0\right)\geq0.
\end{align*}
Thus
\begin{align*}
h_n\left(x_0\right)\geq c_n\left(x_0\right)e^{-u_n\left(x_0\right)}.
\end{align*}
Letting $n\to\infty$, we deduce that
\begin{align*}
h\left(x_0\right)\geq0. 
\end{align*}
The proof is completed.
\end{proof}

Now we can prove the following compactness result.
\begin{theorem}\label{thm:a-priori}
Let $G=(V,E)$ be a connected finite graph with weight $\omega$ and measure $\mu$. Assume that there exists a positive constant $A$ satisfying:
\begin{enumerate}[$(1)$]
    \item $\max_{V}\left(\abs{h}+\abs{c}\right)\leq A$;
    \item if $h(x)>0$ for some $x\in V$, then $h(x)\geq A^{-1}$;
    \item if $c>0$, then $c\geq A^{-1}$;
    \item if $c=0$, then  $\int_{V}h\dif\mu\leq -A^{-1}$;
    \item if $c<0$, then $c\leq-A^{-1}$ and $\min_{V}h\leq -A^{-1}$.
\end{enumerate}
Then there exists a positive constant $C$ depending only on $A$, $G$, $\omega$ and $\mu$ such that every solution to \eqref{eq:KW}
satisfies
\begin{align*}
\max_{x\in V}\abs{u(x)}\leq C.
\end{align*}
\end{theorem}

\begin{proof}

Assume there is a sequence $u_n\in V^{\mathbb{R}}$ of solutions to
\begin{align*}
-\Delta u_n=h_ne^{u_n}-c_n
\end{align*}
satisfying
\begin{align*}
\lim_{n\to\infty}h_n=h,\quad\lim_{n\to\infty}c_n=c,\quad
\lim_{n\to\infty}\norm{u_n}=\infty.
\end{align*}
Here $h_n$ and $c_n$ satisfy the conditions $(1)-(5)$.

If $u_n$ converges  uniformly to $-\infty$,  then
\begin{align*}
-\Delta\left(u_n-\min_{V}u_n\right)=h_ne^{u_n}-c_n.
\end{align*}
which implies that $u_n-\min_{V}u_n$ is uniformly bounded due to \autoref{lem:elliptic}.
Thus $u_n-\min_{V}u_n$ converges uniformly to a solution $w$ of the equation
\begin{align*}
-\Delta w=-c,\quad\min_{V}w=0
\end{align*}
But this then implies that $c=0$ and $w\equiv0$. By assumptions $(3)$ and $(5)$, we may assume $c_n=0$. Then the assumption $(4)$ gives
\begin{align*}
    \int_{V}h_n\dif\mu\leq -A^{-1}.
\end{align*}
Taking $n\to\infty$, we obtain
\begin{align*}
    \int_{V}h\dif\mu\leq -A^{-1}.
\end{align*}
However,
\begin{align*}
    0=e^{-\min_{V}u_n}\int_Vh_ne^{u_n}\dif\mu=\int_Vh_ne^{u_n-\min_{V}u_n}\dif\mu\to\int_{V}h\dif\mu\leq -A^{-1}
\end{align*}
as $n\to\infty$, which is a contradiction.

Applying \autoref{thm:alternative}, we may assume that $\max_{V}u_n$ converges to $+\infty$, and $u_n$ is uniformly bounded from below in $V$, and $u_n$ is uniformly bounded in $\Omega\coloneqq\set{x\in V: h(x)>0}$, and $\set{x\in V:h(x)=0}\neq\emptyset$. For $n$ large, the assumption $(2)$ gives
\begin{align*}
    \Omega\subseteq\set{x\in V: h_n(x)>0}\subseteq\set{x\in V: h_n(x)\geq A^{-1}}\subseteq \set{x\in V: h(x)\geq A^{-1}}\subseteq\Omega,
\end{align*}
and therefore all the set containment are in fact set equalities.
We have
\begin{align*}
    \int_{V}c_n\dif\mu=&\int_{V}h_ne^{u_n}\dif\mu\\
    =&\int_{\Omega}h_ne^{u_n}\dif\mu+\int_{V\setminus\Omega}h_ne^{u_n}\dif\mu\\
    \leq&C-\int_{V}h_n^-e^{u_n}\dif\mu.
\end{align*}
This implies
\begin{align*}
    \int_{V}h_n^{-}e^{u_n}\dif\mu\leq C.
\end{align*}
Hence
\begin{align*}
    \norm{\Delta u_n}_{L^1\left(V\right)}\leq \int_{V}\abs{h_n}e^{u_n}\dif\mu+C\leq C
\end{align*}
According to \autoref{lem:elliptic}, we know that
\begin{align*}
    \max_{V}u_n\leq\min_{V}u_n+C.
\end{align*}
By the assumption that $\max_{V}u_n$ converges to $+\infty$, we deduce that $u_n$ must converge uniformly to $+\infty$. Consequently, $\Omega=\emptyset$, i.e., $h\leq0$. By assumption $(2)$, we may assume $h_n\leq0$. Thus
\begin{align*}
    \int_{V}c_n\dif\mu=\int_{V}h_ne^{u_n}\dif\mu\leq0,
\end{align*}
with the equality if and only if $h_n\equiv0$. If $c_n=0$, then $h_n=0$, contradicting the assumption (4). Hence $c_n<0$. By $(5)$ we know that
\begin{align*}
    \min_{V}h_n\leq-A^{-1}.
\end{align*}
Hence
\begin{align*}
    -CA\leq\int_{V}c_n\dif\mu=\int_{V}h_ne^{u_n}\dif\mu\leq-CA^{-1}e^{\min_{V}u_n}.
\end{align*}
This implies $\min_{V}u_n\leq C$, which is a contradiction.
\end{proof}

\begin{eg}
For every positive number $\varepsilon$, we have
\begin{align*}
    -\Delta\ln\varepsilon=0=\pm\left(e^{\ln\varepsilon}-\varepsilon\right).
\end{align*}
Thus $u=\ln \epsilon$ is a solution to the equation $-\Delta u=he^u-c$ with $h=\pm 1$ and $c=\pm \epsilon$.
Thus the conditions $(1), (3)$ and the first part of condition $(5)$ are necessary since $\lim\limits_{\varepsilon\to0}\ln\varepsilon=-\infty$ and $\lim\limits_{\varepsilon\to+\infty}\ln\varepsilon=+\infty$.

We also have
\begin{align*}
    -\Delta\left(-\ln\varepsilon\right)=\pm\left(\varepsilon e^{-\ln\varepsilon}-1\right),
\end{align*}
which implies that the conditions $(1), (2)$ and the second part of condition $(5)$ are necessary.

Assume $h\not\equiv0$ and $\int_{V}h\dif\mu<0$. The following Kazdan-Warner equation is solvable (see  \cite{GriLinYan16kazdan} or \autoref{cor:flat})
\begin{align*}
    -\Delta u=he^{u}.
\end{align*}
Let $u$ be a solution to the above equation. Then
\begin{align*}
    -\Delta\left(u-\ln\varepsilon\right)=\varepsilon he^{u-\ln\varepsilon},
\end{align*}
which implies that the conditions $(1)$ and $(4)$ are necessary.
\end{eg}

Now we can prove \autoref{main:a-priori}.
\begin{proof}[Proof of \autoref{main:a-priori}]
For fixed $h$ and $c$, it is easy to check that the conditions $(1)-(5)$ in \autoref{thm:a-priori} hold under  assumptions of \autoref{main:a-priori}. Therefore, by \autoref{main:a-priori} we know that every solution to \eqref{eq:KW} is uniformly bounded by  a positive constant  depending only on $h, c$ and $G$.
\end{proof}

\section{Brouwer degree}\label{sec:degree}

In this section, we prove \autoref{main:degree}. We divided it into three cases: $c>0, c=0$ and $c<0$. These cases correspond to \autoref{thm:positive}, \autoref{thm:flat} and \autoref{thm:negative}, respectively.

Firstly, we compute the Brouwer  degree for the positive case: $c>0$. In other words, we prove the following

\begin{theorem}\label{thm:positive}
For every connected finite graph $G=(V,E)$, function $h$ with $\max_{V}h>0$ and $c>0$,  we have $d_{h,c}=-1$.
\end{theorem}
\begin{proof}Let $u_t\in V^{\mathbb{R}}$ satisfy
\begin{align*}
    -\Delta u_t=\left(h^+-(1-t)h^{-}\right)e^{u_t}-(1-t)c-t\varepsilon,\quad t\in[0,1],
\end{align*}
where $\varepsilon>0$ is small to be determined.
According to  \autoref{thm:a-priori}, $u_t$ is uniformly bounded.  By the homotopic invariance  of the Brouwer degree, we may assume $h^-\equiv0$ and $c=\varepsilon>0$.

To compute the Brouwer degree $d_{h,f}$, we may assume $h$ vanishes nowhere in $V$. In fact,  without loss of generality, we may assume $\mu\equiv1$ and $V=\set{1,2,\dotsc, m}$. In this case $L\coloneqq-\Delta=\left(l_{ij}\right)_{m\times m}$ is a symmetric matrix which can be characterized by
\begin{align*}
\begin{cases}
l_{ij}\leq0,\quad \forall i\neq j,\\
\sum_{j}l_{ij}=0,\quad\forall i,
\end{cases}
\end{align*}
and $\dim\ker L=1$ (see \autoref{rem:dim=1}). Since $L$ is a symmetric diagonally dominant real matrix with nonnegative diagonal entries, we know that $L$ is positive semi-definite. This can also be seen from Green's formula \eqref{Green}.

Now the Kazdan-Warner equation \eqref{eq:KW-hf} is equivalent to
\begin{align*}
L\begin{pmatrix}x_1\\
x_2\\
\vdots\\
x_m
\end{pmatrix}=\begin{pmatrix}h_1e^{x_1}-f_1\\
h_2e^{x_2}-f_2\\
\vdots\\
h_me^{x_m}-f_m
\end{pmatrix},
\end{align*}
where $h_i=h(i), f_i=f(i)$.
We also assume $h_1\neq0,\dotsc, h_{r}\neq0, h_{r+1}=\dotsm=h_{m}=0$ and $1\leq r\leq m$. Write
\begin{align*}
L=\begin{pmatrix}P&Q^T\\
Q&R
\end{pmatrix},
\end{align*}
where $P$ is a $r\times r$ matrix.  If $r<m$, then $R$ is positive definite (see \autoref{rem:R-positive} for details). Now \eqref{eq:KW-hf} is equivalent to
\begin{align*}
\left(P-Q^TR^{-1}Q\right)\begin{pmatrix}x_1\\
x_2\\
\vdots\\
x_r
\end{pmatrix}=\begin{pmatrix}h_1e^{x_1}-\tilde f_1\\
h_2e^{x_2}-\tilde f_2\\
\vdots\\
h_re^{x_r}-\tilde f_r
\end{pmatrix},
\end{align*}
where
\begin{align*}
\begin{pmatrix}\tilde f_1\\
\tilde f_2\\
\vdots\\
\tilde f_r
\end{pmatrix}=\begin{pmatrix}f_1\\
f_2\\
\vdots\\
f_r
\end{pmatrix}-Q^TR^{-1}\begin{pmatrix}f_{r+1}\\
f_{r+2}\\
\vdots\\
f_m
\end{pmatrix}.
\end{align*}
One can check that $R^{-1}=\left(r^{ij}\right)$ satisfies $r^{ij}\geq0$ (see \cite[p.\ 137]{Berman}) and $ \tilde L\coloneqq P-Q^{T}R^{-1}Q=\left(\tilde l_{ij}\right)$  satisfies
\begin{align*}
\begin{cases}
\tilde l_{ji}=\tilde l_{ij}\leq0,\quad\forall i\neq j,\\
\sum_{j}\tilde l_{ij}=0,\quad\forall i,
\end{cases}
\end{align*}
and $\dim\ker\tilde L=1$. Thus, we can construct a connected finite graph $\tilde G=(\tilde V,\tilde E)$ with vertex $\tilde V=\set{1,2,\dotsc,r}$, weight $\tilde\omega_{ij}=-\tilde l_{ij}$, measure $\tilde\mu\equiv1$ and $\tilde L=-\tilde\Delta$.  Moreover,
\begin{align*}
\sum_{i=1}^r\tilde f_i=\sum_{i=1}^mf_i,
\end{align*}
and
\begin{align*}
    \det\left(L-\mathrm{diag}\left(h_1e^{x_1},\dotsc,h_me^{x_m}\right)\right)=\det R \cdot \det\left(\tilde L-\mathrm{diag}\left(h_1e^{x_1},\dotsc,h_re^{x_r}\right)\right).
\end{align*}
We conclude that
\begin{align*}
    d_{h,f}=d_{h\vert_{\tilde V}, \tilde f}.
\end{align*}

Applying \autoref{thm:a-priori} again, by the homotopic invariance  of the Brouwer degree, we may assume $h\equiv1$ and $\mu\equiv1$.

We consider
\begin{align}\label{eq:kze}
-\Delta u(x)=e^{u(x)}-\varepsilon,\quad x\in V.
\end{align}
Notice that
\begin{align*}
\int_{V}e^{u}\dif\mu=\varepsilon\int_{V}1\dif\mu.
\end{align*}
We obtain
\begin{align*}
e^{\max_{V}u}\leq C\varepsilon.
\end{align*}
If $w$ solves
\begin{align*}
-\Delta w=e^{w}-\varepsilon,
\end{align*}
then
\begin{align*}
-\Delta(u-w)=e^{u}-e^{w},
\end{align*}
which implies
\begin{align*}
\abs{\Delta(u-w)}\leq \max\set{e^{u},e^{w}}\abs{u-w}\leq C\varepsilon\abs{u-w}.
\end{align*}
Thus for small $\varepsilon>0$, applying \autoref{lem:elliptic}, we have
\begin{align*}
\max_{V}\left(u-w\right)\leq \min_{V}\left(u-w\right)+\dfrac12\abs{\min_{V}\left(u-w\right)}.
\end{align*}
If $u\not\equiv w$, then  $e^{u}-e^{w}\not\equiv0$. Since
\begin{align*}
    \int_{V}\left(e^{u}-e^{w}\right)\dif\mu=0,
\end{align*}
we must have
\begin{align*}
    \min_{V}\left(u-w\right)<0<\max_{V}\left(u-w\right).
\end{align*}
Hence
\begin{align*}
0<\max_{V}\left(u-w\right)\leq\min_{V}\left(u-w\right)-\dfrac12\min_{V}\left(u-w\right)=\dfrac12\min_{V}\left(u-w\right)<0,
\end{align*}
which is a contradiction. Therefore, the Kazdan-Warner equation \eqref{eq:kze} has a unique solution $u=\ln\varepsilon$ if $\varepsilon>0$ is small.

Note that $-\Delta$ is a nonnegative matrix and $0$ is an eigenvalue of $-\Delta$ with multiplicity one. We have for small $\varepsilon>0$,
\begin{align*}
\det\left(DF_{1,\varepsilon}\left(\ln\varepsilon\right)\right)=\det\left(-\Delta-\varepsilon\mathrm{Id}\right)<0,
\end{align*}
Consequently, by the homotopy invariance of the Brouwer degree,
\begin{align*}
d_{h,c}=\lim_{\varepsilon\searrow 0}d_{1,\varepsilon}=\mathrm{sgn}\det\left(DF_{1,\varepsilon}\left(\ln\varepsilon\right)\right)=-1
\end{align*}
and the proof is finished.
\end{proof}
\begin{rem}\label{rem:dim=1}
The condition $\dim \ker L=1$ is equivalent to that the graph $G=(V,E)$ is connected, i.e., has only one connected component. This further guarntees that no rows of $L$ can be zero and therefore $l_{ii}>0$ for each $i$.
\end{rem}

\begin{rem}\label{rem:R-positive}
Here we give details explaining why $R$ is positive definite. First, since $L$ is positive semi-definite, we know that $R$ is also positive semi-definite. Therefore, all the eigenvalues of $R$ are nonnegative. If $R$ is not positive, then it has an eigenvalue 0. Let $Y$ be an eigenvector for the eigenvalue 0. For any column vector $X\in \mathbb{R}^{m-r}$, let $Z=\begin{pmatrix}
X \\ Y
\end{pmatrix}$. We have
\begin{align*}
Z^\mathrm{T}LZ=(X^\mathrm{T},Y^\mathrm{T})\begin{pmatrix}
P & Q^\mathrm{T} \\
Q & R
\end{pmatrix}\begin{pmatrix}
X \\ Y
\end{pmatrix}=X^\mathrm{T}PX+2X^\mathrm{T}Q^\mathrm{T}Y.
\end{align*}
Since $L$ is positive semi-definite, we have $X^\mathrm{T}PX+2X^\mathrm{T}Q^TY\geq 0$ for any $X\in \mathbb{R}^{m-r}$. Therefore, we must have $Q^\mathrm{T}Y=0$. Note that
\begin{align*}
LZ=\begin{pmatrix}
PX \\ QX
\end{pmatrix}.
\end{align*}
It turns out that when $Y\neq 0$, $\left(0,Y^{\mathrm{T}}\right)^\mathrm{T}$ is an eigenvector of $L$ of the eigenvalue 0. However, by the definition of $L$, we know that  $(1,1,\cdots, 1)^\mathrm{T}$ is an eigenvector of $L$ corresponding to the eigenvalue 0. Since we have assumed that $\dim\ker L=1$, we have
\begin{align*}
\ker L=\mathrm{span}(1,1,\cdots, 1)^\mathrm{T}.
\end{align*}
This implies that $Y=0$, which contradicts the assumption that $Y$ is an eigenvector of $R$.
\end{rem}

Secondly, we compute the Brouwer degree for the flat case: $c=0$. We prove the following
\begin{theorem}\label{thm:flat}
  For every connected finite graph $G=(V,E)$ and sign changed function $h\in V^{\mathbb{R}}$ with $  \int_{V}h\dif\mu<0$, we have
 \begin{align*}
 d_{h,0}=-1.
 \end{align*}
\end{theorem}
\begin{proof}
Let $u_{t}\in V^{\mathbb{R}}$ be a solution to
\begin{align*}
    -\Delta u_{t}=he^{u_{t}}-t,\quad  t\in[0,1].
\end{align*}
We claim that there exists a positive constant $C$ such that
\begin{align}\label{eq:apriori-u_t}
    \max_{V}\abs{u_{t}}\leq C,\quad\forall t\in[0,1].
\end{align}
It suffices to prove
\begin{align}\label{add-ineq}
    \limsup_{t\searrow0}\max_{V}\abs{u_{t}}\leq C.
\end{align}
We prove it by contradiction. Suppose \eqref{add-ineq} is not true. According to \autoref{thm:alternative}, after passing to a subsequence, there are two cases:
\begin{enumerate}
    \item $u_{t}$ converges uniformly to $-\infty$, or
    \item $u_{t}$ is uniformly bounded from below in $V$ and  uniformly bounded in $\set{x\in V:h(x)>0}\neq\emptyset$.
\end{enumerate}

For the first case,  arguing similarly as in the beginning of the proof of \autoref{thm:a-priori}, we know that $u_{t}-\min_{V}u_{t}$ converges uniformly to $0$. We conclude that
\begin{align*}
    0>\int_{V}h\dif\mu=\lim_{t\searrow0}\int_{V}he^{u_{t}-\min_{V}u_{t}}\dif\mu=\lim_{t\searrow0}\int_{V}t e^{-\min_{V}u_{t}}\dif\mu\geq0,
\end{align*}
which is a contradiction.

For the second case, we have
\begin{align*}
    \int_{V}\abs{h}e^{u_{t}}\dif\mu=&\int_{\set{x\in V:h(x)>0}}he^{u_{t}}\dif\mu-\int_{\set{x\in V:h(x)<0}}he^{u_{t}}\dif\mu\\
    =&2\int_{\set{x\in V:h(x)>0}}he^{u_{t}}\dif\mu-\int_{V}he^{u_{t}}\dif\mu\\
    =&2\int_{\set{x\in V:h(x)>0}}he^{u_{t}}\dif\mu-\int_{V}t\dif\mu\\
    \leq&C.
\end{align*}
This implies
\begin{align*}
    \min_{V}u_{t}\leq C.
\end{align*}
Moreover, it also implies
\begin{align*}
\max_{V}|\Delta u_{t}|\leq \max_{V} |h|e^{u_t} +1  \leq C.
\end{align*}
Applying  \autoref{lem:elliptic}, we have
\begin{align*}
    \max_{V}u_{t}\leq\min_{V}u_{t}+C\max_{V}|\Delta u|\leq\min_{V}u_{t}+C\leq C,
\end{align*}
which is a contradiction.

The a priori estimate \eqref{eq:apriori-u_t}  implies
\begin{align*}
    d_{h,0}=\lim_{t\searrow0}d_{h,t}=-1.
\end{align*}
Here we used the homotopic invariance of the Brouwer degree and \autoref{thm:positive}.
\end{proof}

 Finally, we compute the Brouwer degree for the negative case: $c<0$. We prove the following
\begin{theorem}\label{thm:negative}
 For every connected finite graph $G=(V,E)$,  function $h$ with $\min_{V}h<0$ and $c<0$, we have
 \begin{align*}
 d_{h,c}=\begin{cases}
 1,&\max_{V}h\leq0,\\
 0,&\max_{V}h>0.
 \end{cases}
 \end{align*}
\end{theorem}
\begin{proof}
Since $c<0$, a necessary condition for the existence to the Kazdan-Warner equation \eqref{eq:KW} is
\begin{align*}
\int_{V}h\dif\mu<0.
\end{align*}
This was proved by Grigor'yan, Lin and Yang \cite[Theorem 3]{GriLinYan16kazdan}. For the sake of completeness, we reproduce its proof here.
In fact, if $u$ solves \eqref{eq:KW}, then
\begin{equation}\label{add-use-Green}
\begin{split}
\int_{V}h\dif\mu=&\int_{V}ce^{-u}\dif\mu-\int_{V}e^{-u}\Delta u\dif\mu \\
<&-\int_{V}e^{-u}\Delta u\dif\mu \quad (\text{since $c<0$}) \\
=&\dfrac12\sum_{x,y\in V}\omega_{x,y}\left(u(x)-u(y)\right)\left(e^{-u(x)}-e^{-u(y)}\right) \quad (\text{by \eqref{Green}}) \\
\leq&0.
\end{split}
\end{equation}

First we assume $\max_{V}h\leq0$. Let $u_t\in V^{\mathbb{R}}$ satisfy
\begin{align*}
    -\Delta u_t=\left(-t+(1-t)h\right)e^{u_t}-c,\quad t\in[0,1],
\end{align*}
According to \autoref{thm:alternative}, $u_t$ is uniformly bounded. By the homotopy invariance of the Brouwer degree, we may assume $h\equiv-1$ and $\mu\equiv1$. We claim that $u\equiv\ln(-c)$ is the unique solution to
\begin{align*}
-\Delta u=-e^{u}-c.
\end{align*}
It suffices to prove that $u$ is a constant function. For otherwise, applying \autoref{lem:maximum}, we have
\begin{align*}
-e^{\max_{V}u}-c>0,\quad -e^{\min_{V}u}-c<0
\end{align*}
which is a contradiction. Thus
\begin{align*}
d_{h,c}=d_{-1,c}=\mathrm{sgn}\det\left(DF_{-1,c}\left(\ln(-c)\right)\right)=\mathrm{sgn}\det\left(-\Delta-c\mathrm{Id}\right)=1.
\end{align*}

Second, we assume $\max_{V}h>0$. Let $v_t\in V^{\mathbb{R}}$ satisfy
\begin{align*}
    -\Delta v_t=\left(th_{\Lambda}+(1-t)h\right)e^{v_t}-c,\quad t\in[0,1],
\end{align*}
where
\begin{align*}
    h_{\Lambda}(x)=\begin{cases}
    \Lambda, & h(x)>0,\\
    -1, & h(x)\leq0,
    \end{cases}
\end{align*}
 and $\Lambda>0$ is large to be determined.
According to \autoref{thm:alternative}, $v_t$ is uniformly bounded.  By the homotopy invariance of the Brouwer degree, without loss of generality, assume $h\equiv h_{\Lambda}$.  Choose $\Lambda$ large such that
\begin{align*}
\int_{V}h_{\Lambda}\dif\mu=\Lambda\int_{\set{x\in V : h(x)>0}}1\dif\mu-\int_{\set{x\in V : h(x)\leq0}}1\dif\mu>0.
\end{align*}
Consequently, there is no solution if $\Lambda$ is large. Thus, according to \autoref{thm:alternative}, by the homotopy invariance and Kronecker existence of the Brouwer degree,
\begin{align*}
d_{h,c}=\lim_{\Lambda\to+\infty}d_{h_{\Lambda},c}=0. 
\end{align*}
We finish the proof.
\end{proof}

\section{Existence results}\label{sec:existence}

As consequences of the degree theory, we state and prove several existence results in the literature. Specifically speaking, we give  proofs of  \autoref{main:cor1}  and \autoref{main:cor2}.
While some proofs and techniques are borrowed from the literature, the main new ingredient in our proofs  is that we apply the degree theory to analyze the existence of solutions.

\begin{cor}[Solvability for positive case]\label{cor:positive}
If $c>0$, then \eqref{eq:KW} is solvable if and only if $\max_{V}h>0$.
\end{cor}

\begin{proof}If there is a solution $u$, since $\int_V he^u \dif\mu=\int_V c\dif\mu>0$, we must have $\max_{V}h>0$.  Under the assumption $\max_{V}h>0$, by \autoref{main:degree} we have $d_{h,c}=-1$. In particular, \eqref{eq:KW} has at least one solution.
\end{proof}

\begin{cor}[Solvability for flat case]\label{cor:flat} If $c=0$ and $h\not\equiv0$, then \eqref{eq:KW} is solvable if and only if $h$ changes sign and $\int_{V}h \dif \mu<0$.
\end{cor}

\begin{proof}
If \eqref{eq:KW} has a soluiton $u$, then since $\int_{V} he^{u}\dif \mu=0$, we know that $h$ must change sign. Moreover, similar to \eqref{add-use-Green}, we have
\begin{align}
\int_{V} h\dif \mu=\frac{1}{2}\sum_{x,y\in V}\omega_{x,y}(u(x)-u(y))\left(e^{-u(x)}-e^{-u(y)} \right)\leq 0,
\end{align}
and the equality holds if and only if $u(x)=u(y)$ whenever $\omega_{x,y}>0$. Since $V$ is connected, this can happen only when $u$ is a constant. But then $h\equiv 0$, contradicting the assumption. Hence we have $\int_{V} h \dif \mu<0$. We remark that the proof here follows the lines of \cite[p.\  92]{GriLinYan16kazdan}.

Conversely, under the assumption that $\int_{V}h\dif\mu<0$ and $h$ changes sign, by \autoref{main:degree} we have $d_{h,0}=-1$. As a consequence,  \eqref{eq:KW} has at least one solution.
\end{proof}

In the rest of this section, we consider the negative case. First, we have the following

\begin{lem}\label{lem:add-negative}
If $c<0$, then \eqref{eq:KW} has a solution if and only if there is a super-solution to \eqref{eq:KW}.
\end{lem}

\begin{proof}
When $A$ is large enough,  the constant function $-A$ satisfies
\begin{align*}
    -\Delta\left(-A\right)+c-he^{-A}=c-he^{-A}<0.
\end{align*}
Thus $-A$ is a sub-solution to \eqref{eq:KW} since $c<0$. Applying the sub- and super-solutions method (\autoref{lem:sub-super}), we complete the proof.
\end{proof}

\begin{cor}\label{cor:-existence-negative}
Assume $c<0$ and $\int_{V}h\dif\mu<0$.
\begin{enumerate}[$(1)$]
\item If $h\leq0$, then \eqref{eq:KW} has a unique (strict  global  minimum) solution.
\item If $\max_{V}h>0$, then there exists a constant $c_h\in(-\infty,0)$ such that  \eqref{eq:KW} has at least two distinct  solutions for $c_h<c<0$, at least a (stable) solution for $c=c_h$ and no solution for $c<c_h$.
\end{enumerate}
\end{cor}

\begin{proof}

If $h\leq0$, we have $d_{h,c}=1$ for every $c<0$. Consequently,  the Kazdan-Warner equation \eqref{eq:KW} is solvable.  Since $h\leq0$, every solution to \eqref{eq:KW} is stable. Applying the strong maximum principle, one conclude that the solution to \eqref{eq:KW} is unique. The fact $d_{h,c}=1$ then implies that the unique solution to \eqref{eq:KW} is the strict global minimum of $J_{h,c}$.

From now on, we assume $\max_{V}h>0$.

 Applying the sub- and super-solution method, Grigor'yan, Lin and Yang \cite{GriLinYan16kazdan} proved that \eqref{eq:KW} is solvable  for every $ c\in[c_0,0)$ where $-c_0>0$ is small. In fact, arguing in a way similar to that of Grigor'yan, Lin and Yang \cite[p.\ 10]{GriLinYan16kazdan}, solve
\begin{align*}
-\Delta v=h-\dfrac{\int_{V}h\dif\mu}{\int_{V}1\dif\mu},\quad\int_{V}v\dif\mu=0.
\end{align*}
For constants $a>0, b=\ln a$, we compute
\begin{align*}
-\Delta\left(av+b\right)=&a\left(h-\dfrac{\int_{V}h\dif\mu}{\int_{V}1\dif\mu}\right)\\
=& he^{av+b}-ah\left(e^{av}-1\right)-a\dfrac{\int_{V}h\dif\mu}{\int_{V}1\dif\mu}\\
\geq&he^{av+b}-a\left(\abs{h}\abs{e^{av}-1}+\dfrac{\int_{V}h\dif\mu}{\int_{V}1\dif\mu}\right).
\end{align*}
Choose $a$ and $-c$ small  to obtain a super-solution $\bar u=av+b$. Thus there exists some $c_1<0$ such that \eqref{eq:KW} has a solution $u_{c_1}$ for $c=c_1$. For any $c\in[c_1,0)$, it is easy to see that
\begin{align*}
    \Delta u_{c_1}+he^{u_{c_1}}-c=c_1-c\leq 0.
\end{align*}
This means that $u_{c_1}$ is a super solution for \eqref{eq:KW}, and hence \eqref{eq:KW} has a solution by \autoref{lem:add-negative}. Let
\begin{align*}
    c_h=\inf\set{c\in \mathbb{R}: \text{\eqref{eq:KW} has a solution}}.
\end{align*}
Then $c_{h}\in[-\infty,0)$ and \eqref{eq:KW} has a  solution if $c\in(c_h,0)$ and no solution if $c<c_{h}$.

 Moreover, \eqref{eq:KW} has a strict local minimum solution for $c\in(c_h,0)$. In fact, following the idea of Liu and Yang \cite{LiuYan20multiple}, let $u_0\in V^{\mathbb{R}}$ satisfy
\begin{align*}
-\Delta u_0=he^{u_0}-c_0>he^{u_0}-c,
\end{align*}
 where $c_0\in(c_h,0)$.
Choose $A>0$ large such that $u_0>-A$ and
\begin{align*}
-\Delta\left(-A\right)<he^{-A}-c.
\end{align*}
Choose $u\in V^{\mathbb{R}}$ such that $-A\leq u\leq u_0$ and
\begin{align*}
J_{h,c}(u)=\min_{-A\leq v\leq u_0}J_{h,c}(v).
\end{align*}
Applying \autoref{lem:maximum}, one can prove that $-A<u<u_0$ and conclude that $u$ is a local minimum  of $J_{h,c}$ (cf. \cite{Ge17kazdan}). Moreover, $u$ is a stict local minimum. In fact, following the lines of \cite[p. 10-11]{LiuYan20multiple}, if there exists $0\not\equiv\xi\in V^{\mathbb{R}}$ such that $\frac{\dif^2}{\dif t^2}\left.J_{h,c}\left(u+t\xi\right)\right\rvert_{t=0}=0$, then
\begin{align*}
-\Delta\xi=he^{u}\xi.
\end{align*}
This implies that $\xi$ is not a constant function since $h\not\equiv0$.
Since $u$ is  a local minimum of $J_{h,c}$, analyzing the Taylor expansion of $J_{h,c}(u+t\xi)$ at the point $t=0$, we deduce that
\begin{align*}
\frac{\dif^3}{\dif t^3}\left.J_{h,c}\left(u+t\xi\right)\right\rvert_{t=0}=0,\quad \frac{\dif^4}{\dif t^4}\left.J_{h,c}\left(u+t\xi\right)\right\rvert_{t=0}\geq0.
\end{align*}
However,
\begin{align*}
\frac{\dif^4}{\dif t^4}\left.J_{h,c}\left(u+t\xi\right)\right\rvert_{t=0}=&-\int_{V}he^{u}\xi^4\dif\mu\\
=&\int_{V}\xi^3\Delta \xi\dif\mu\\
=&-\dfrac12\sum_{x,y}\omega_{xy}\left(\xi(x)-\xi(y)\right)\left(\xi^3(x)-\xi^3(y)\right)\\
<&0
\end{align*}
which is a contradiction. In other words, $u$ is strictly stable which implies that $u$ is a strict local minimum.

 If $h\leq0$ and $\min_{V}h<0$, then we conclude that $c_{h}=-\infty$ since \eqref{eq:KW} is solvable for every $c<0$.  If $c_{h}=-\infty$, then Ge \cite{Ge17kazdan} proved that $h\leq0$. In fact, if $\max_{V}h>0$, then
 \begin{align}\label{eq:c_h}
     c_h\geq-\dfrac{C\norm{\Delta h}}{\max_{V}h^+}.
 \end{align}
Following the lines of Ge \cite{Ge17kazdan}, assume $u_{c},\xi_c$ satisfies
 \begin{align*}
     -\Delta u_c=he^{u_c}-c,\quad\left(\Delta+c\right)\xi_c=h.
 \end{align*}
 Notice that for every $x\in V$,
 \begin{equation}\label{add-strict-ineq}
     \begin{split}
         -e^{-u_c(x)}\Delta u_c(x)=&\dfrac{1}{\mu_{x}}\sum_{y\in V}\omega_{xy}\left(u_c(x)-u_c(y)\right)e^{-u_c(x)}\\
     \leq&\dfrac{1}{\mu_{x}}\sum_{y\in V}\omega_{xy}\left(e^{-u_c(y)}-e^{-u_c(x)}\right) \\
     =&\Delta e^{-u_c}(x).
     \end{split}
 \end{equation}
 Here we used the inequality $e^t-1\geq t$ ($\forall t\in \mathbb{R}$) wherein the equality holds if and only if $t=0$.  We have
 \begin{align*}
     \left(\Delta+c\right) e^{-u_c}\geq-e^{-u_c}\Delta u_c+ce^{-u_c}=h=\left(\Delta+c\right)\xi_c,
 \end{align*}
 and the strict inequality holds at some point since the inequality in \eqref{add-strict-ineq} cannot always be equality as $u_c$ is not a constant function. Let $g=e^{-u_c}-\xi_c$. Then
 \begin{align}\label{add-Delta-ineq}
     \Delta g\geq -cg.
 \end{align} If $g$ is a constant, then we immediately deduce that $g<0$. If $g$ is not a constant, by \autoref{lem:maximum} we may choose some $x_1\in V$ such that
 \begin{align*}
      g(x_1)=\max_V g, \quad \Delta g(x_1)<0.
 \end{align*}
 This together with \eqref{add-Delta-ineq} imply $g<0$. In other words,  we have proved that
 \begin{align*}
     \xi_{c}>e^{-u_c}.
 \end{align*}
 Hence when $-c$ is large enough,
 \begin{align*}
     0>c\xi_{c}=\left(1+c^{-1}\Delta\right)^{-1}h=h-c^{-1}\Delta h+O\left(c^{-2}\norm{\Delta h}\right).
 \end{align*}
 We obtain the desired estimate \eqref{eq:c_h}. As a consequence, if  $c_h=-\infty$, then $\max_{V}h\leq0$.

If $c_{h}>-\infty$, then $\max_{V}h>0$ and $d_{h,c}=0$. We have already proved that there exists a strict local minimum solution for every $c_h<c<0$. Hence, there must be another solution for $c_h<c<0$.

If $c_{h}>-\infty$, then we want to prove that there exists a stable solution for $c=c_{h}$. Let $u_{c}$ be a strict local minimum solution for each $c_h<c<0$.  According to \autoref{thm:alternative}, we know that $u_{c}$ is uniformly  bounded. Thus, letting $c\searrow c_{h}$, after passing to a subsequence, we obtain a solution to \eqref{eq:KW} for $c=c_{h}$. Since $u_c$ is stable, we conclude that the limit is also stable.

Up to now, we obtain a stable solution for each $c\in[c_h,0)$. Moreover,  \eqref{eq:KW} has a strict local minimum solution for every $c\in(c_h,0)$. Since the Brouwer degree $d_{h,c}=0$ under the assumption $\max_{V}h>0$, we conclude that \eqref{eq:KW} has at least two distinct solutions for every $c\in(c_h,0)$.
\end{proof}

Combining Corollaries \ref{cor:positive}, \ref{cor:flat} and \ref{cor:-existence-negative}, we complete the proof of \autoref{main:cor1}.

Finally, we provide a new proof of \autoref{main:cor2} different from that in \cite{LiuYan20multiple}.

\begin{proof}[Proof of \autoref{main:cor2}]Without loss of generality, assume $\kappa\equiv-1$. Recall the condition $\min_V K<\max_V K=0$ for \eqref{eq:K-lambda}. Since $K_{\lambda}\leq0$ for $\lambda\leq0$, according to \autoref{main:cor1}, we conclude that  \eqref{eq:K-lambda} has only one (strict global minimum) solution  when $\lambda\leq0$. In particular, when $\lambda=0$, there is a strict global minimum of $J_{K,\kappa}$.

Let $\psi$ be the unique solution to
\begin{align*}
    -\Delta\psi=Ke^{\psi}-\kappa+1.
\end{align*}
Then
\begin{align*}
   -\Delta\psi+\kappa-K_{\lambda}e^{\psi}=1-\lambda e^{\psi}.
\end{align*}
Thus for small $\lambda$, we obtain a super-solution $\psi$ to \eqref{eq:K-lambda}. Applying the sub- and super-solutions method, we conclude that \eqref{eq:K-lambda} is solvable for small $\lambda$. Define
\begin{align*}
\lambda^*=\sup\set{\lambda\in\mathbb{R}: \text{\eqref{eq:K-lambda} has a solution}}.
\end{align*}
If \eqref{eq:K-lambda} has a solution, then
\begin{align*}
\int_V K_\lambda e^u\dif \mu=\int_V \kappa \dif \mu<0.
\end{align*}
This implies $\min_V K_\lambda =\min_V K+\lambda<0$, i.e., $\lambda<-\min_V K$. Hence we have $\lambda^*\leq-\min_{V}K$.

Applying the sub- and super-solutions principle, if \eqref{eq:K-lambda} has a solution when $\lambda=\lambda_0$, then  \eqref{eq:K-lambda} has a solution for every $\lambda<\lambda_0$.
 One can check that there exists a strict local minimum $u_{\lambda}$ of $J_{K_\lambda,\kappa}$ for $\lambda<\lambda^*$. Since the Brouwer degree $d_{K_{\lambda},\kappa}=0$ for $0<\lambda<\lambda^*$, we conclude that there exists another solution to \eqref{eq:K-lambda}.  By definition, \eqref{eq:K-lambda} has no solution for any $\lambda>\lambda^*$.

Consider the sequence $\set{u_{\lambda}}_{0<\lambda<\lambda^*}$. We prove that $u_{\lambda}$ is uniformly bounded to complete the proof. For otherwise, according to \autoref{thm:alternative}, since $\int_{V}\kappa\dif\mu<0$, we may assume $\max_{V}u_{\lambda}$ converges to $+\infty$, and  $u_{\lambda}$ is uniformly bounded from below in $V$,  and $u_{\lambda}$ is uniformly bounded in $\set{x\in V: K_{\lambda^*}(x)>0}\neq\emptyset$. By definition $K_{\lambda}=K+\lambda$. Thus, for $\lambda^*-\lambda$ small, we conclude that $u_{\lambda}$ is uniformly bounded in $\set{x\in V:K_{\lambda}(x)>0}$. Therefore,
\begin{align*}
    \int_{V}\kappa\dif\mu=&\int_{V}K_{\lambda}e^{u_{\lambda}}\dif\mu\\
    =&\int_{\set{x\in V: K_{\lambda}(x)>0}}K_{\lambda}e^{u_{\lambda}}\dif\mu+\int_{\set{x\in V: K_{\lambda}(x)\leq0}}K_{\lambda}e^{u_{\lambda}}\dif\mu\\
    =&\int_{\set{x\in V: K_{\lambda^*}(x)>0}}K_{\lambda}e^{u_{\lambda}}\dif\mu+\int_{\set{x\in V: K_{\lambda}(x)\leq0}}K_{\lambda}e^{u_{\lambda}}\dif\mu\\
    \leq&C-\int_{V}K_{\lambda}^-e^{u_{\lambda}}\dif\mu.
\end{align*}
This implies $\int_V K_\lambda^{-}e^{u_\lambda}\dif \mu\leq C$. The second line of the above equation also implies
\begin{align*}
  \int_V K_\lambda^{+}e^{u_\lambda}\dif \mu=\int_V\kappa \dif \mu+\int_V K_\lambda^{-}\dif \mu.
\end{align*}
Hence
\begin{align*}
    \int_{V}K_{\lambda}^+e^{u_{\lambda}}\dif\mu<\int_{V}K_{\lambda}^-e^{u_{\lambda}}\dif\mu\leq C.
\end{align*}
Thus
\begin{align*}
    \norm{\Delta u_{\lambda}}_{L^1\left(V\right)}\leq C.
\end{align*}
We obtain
\begin{align*}
    \max_{V}u_{\lambda}\leq\min_{V}u_{\lambda}+C\leq C
\end{align*}
which is a contradiction.
\end{proof}


\begin{thebibliography}{10}

\bibitem{Berman}
A.~Berman and R.~J. Plemmons.
\newblock {\em Nonnegative matrices in the mathematical sciences}, volume~9 of
  {\em Classics in Applied Mathematics}.
\newblock Society for Industrial and Applied Mathematics (SIAM), Philadelphia,
  PA, 1994.
\newblock Revised reprint of the 1979 original.

\bibitem{BreMer91uniform}
H.~Brezis and F.~Merle.
\newblock Uniform estimates and blow-up behavior for solutions of {$-\Delta
  u=V(x)e^u$} in two dimensions.
\newblock {\em Comm. Partial Differential Equations}, 16(8-9):1223--1253, 1991.

\bibitem{CafYan95vortex}
L.~A. Caffarelli and Y.~S. Yang.
\newblock Vortex condensation in the {C}hern-{S}imons {H}iggs model: an
  existence theorem.
\newblock {\em Comm. Math. Phys.}, 168(2):321--336, 1995.

\bibitem{Cha05methods}
K.-C. Chang.
\newblock {\em Methods in nonlinear analysis}.
\newblock Springer Monographs in Mathematics. Springer-Verlag, Berlin, 2005.

\bibitem{ChaYan87prescribing}
S.-Y.~A. Chang and P.~C. Yang.
\newblock Prescribing {G}aussian curvature on {$S^2$}.
\newblock {\em Acta Math.}, 159(3-4):215--259, 1987.

\bibitem{CheLin03topological}
C.-C. Chen and C.-S. Lin.
\newblock Topological degree for a mean field equation on {R}iemann surfaces.
\newblock {\em Comm. Pure Appl. Math.}, 56(12):1667--1727, 2003.

\bibitem{CheDin87scalar}
W.~X. Chen and W.~Y. Ding.
\newblock Scalar curvatures on {$S^2$}.
\newblock {\em Trans. Amer. Math. Soc.}, 303(1):365--382, 1987.

\bibitem{DinJosLiWan99existence}
W.~Ding, J.~Jost, J.~Li, and G.~Wang.
\newblock Existence results for mean field equations.
\newblock {\em Ann. Inst. H. Poincar\'{e} Anal. Non Lin\'{e}aire},
  16(5):653--666, 1999.

\bibitem{DinLiu95note}
W.~Y. Ding and J.~Q. Liu.
\newblock A note on the problem of prescribing {G}aussian curvature on
  surfaces.
\newblock {\em Trans. Amer. Math. Soc.}, 347(3):1059--1066, 1995.

\bibitem{Dja08existence}
Z.~Djadli.
\newblock Existence result for the mean field problem on {R}iemann surfaces of
  all genuses.
\newblock {\em Commun. Contemp. Math.}, 10(2):205--220, 2008.

\bibitem{Ge17kazdan}
H.~Ge.
\newblock Kazdan-{W}arner equation on graph in the negative case.
\newblock {\em J. Math. Anal. Appl.}, 453(2):1022--1027, 2017.

\bibitem{GeJia18Kazdan}
H.~Ge and W.~Jiang.
\newblock Kazdan-{W}arner equation on infinite graphs.
\newblock {\em J. Korean Math. Soc.}, 55(5):1091--1101, 2018.

\bibitem{GriLinYan16kazdan}
A.~Grigor'yan, Y.~Lin, and Y.~Yang.
\newblock Kazdan-{W}arner equation on graph.
\newblock {\em Calc. Var. Partial Differential Equations}, 55(4):Art. 92, 13,
  2016.

\bibitem{HuaLinYau20existence}
A.~Huang, Y.~Lin, and S.-T. Yau.
\newblock Existence of solutions to mean field equations on graphs.
\newblock {\em Comm. Math. Phys.}, 377(1):613--621, 2020.

\bibitem{KazWar74curvature}
J.~L. Kazdan and F.~W. Warner.
\newblock Curvature functions for compact {$2$}-manifolds.
\newblock {\em Ann. of Math. (2)}, 99:14--47, 1974.

\bibitem{KelSch18kazdan}
M.~Keller and M.~Schwarz.
\newblock The {K}azdan-{W}arner equation on canonically compactifiable graphs.
\newblock {\em Calc. Var. Partial Differential Equations}, 57(2):Paper No. 70,
  18, 2018.

\bibitem{Li99harnack}
Y.~Y. Li.
\newblock Harnack type inequality: the method of moving planes.
\newblock {\em Comm. Math. Phys.}, 200(2):421--444, 1999.

\bibitem{LiuYan20multiple}
S.~Liu and Y.~Yang.
\newblock Multiple solutions of {K}azdan-{W}arner equation on graphs in the
  negative case.
\newblock {\em Calc. Var. Partial Differential Equations}, 59(5):Paper No. 164,
  15, 2020.

\bibitem{Nol03nontopological}
M.~Nolasco.
\newblock Nontopological {$N$}-vortex condensates for the self-dual
  {C}hern-{S}imons theory.
\newblock {\em Comm. Pure Appl. Math.}, 56(12):1752--1780, 2003.

\bibitem{RicTar00vortices}
T.~Ricciardi and G.~Tarantello.
\newblock Vortices in the {M}axwell-{C}hern-{S}imons theory.
\newblock {\em Comm. Pure Appl. Math.}, 53(7):811--851, 2000.

\bibitem{StruTar98multivortex}
M.~Struwe and G.~Tarantello.
\newblock On multivortex solutions in {C}hern-{S}imons gauge theory.
\newblock {\em Boll. Unione Mat. Ital. Sez. B Artic. Ric. Mat. (8)},
  1(1):109--121, 1998.

\end{thebibliography}

\end{document}